\documentclass[12pt]{amsart}
\usepackage{amssymb}
\usepackage{enumerate}
\usepackage{graphicx}
\usepackage[T1]{fontenc}
\usepackage{amsfonts}
\usepackage{cite}
\usepackage{amsmath}
\usepackage{amsfonts,amssymb,latexsym}
\usepackage[cp1250]{inputenc}
\usepackage{mathtools}
\usepackage{verbatim}
\usepackage{centernot}

\makeatletter
\@namedef{subjclassname@2010}{  \textup{2010} Mathematics Subject Classification}
\makeatother
\newtheorem{theorem}{Theorem}[section]

\newtheorem{lemma}[theorem]{Lemma}

\theoremstyle{definition}
\newtheorem{definition}[theorem]{Definition}
\newtheorem{remark}[theorem]{Remark}

\numberwithin{equation}{section}
\DeclareMathOperator{\Fix}{Fix}

\DeclarePairedDelimiter\norm{\Vert}{\Vert}

\mathchardef\mhyp="2D

\begin{document}
\title[Fixed point theorems]{Common fixed point theorems for nonexpansive
mappings using the lower semicontinuity property}
\author[S. Borzdy\'{n}ski]{S\l awomir Borzdy\'{n}ski}
\address{S\l awomir Borzdy\'{n}ski, Institute of Mathematics, Maria Curie-Sk%
\l odowska University, 20-031 Lublin, Poland}
\email{slawomir.borzdynski@gmail.com}
\date{}

\begin{abstract}
Suppose that $E$ is a Banach space, $\tau $ a topology under which the norm
of $E$ becomes $\tau $-lower semicontinuous and $\mathcal{S}$ a commuting
family of $\tau $-continuous nonexpansive mappings defined on a $\tau $%
-compact convex subset $C$ of $E.$ It is shown that the set of common fixed
points of $\mathcal{S}$ is a nonempty nonexpansive retract of $C$. Along the
way, a few other related fixed point theorems are derived.
\end{abstract}

\maketitle


\baselineskip=17pt



\section{Introduction}

Let $C$ be a subset of a Banach space $E$. A mapping $T:C\rightarrow C$ is
said to be nonexpansive if $\left\Vert Tx-Ty\right\Vert \leq \left\Vert
x-y\right\Vert $ for every $x,y\in C.$ In this paper we shall study fixed
point properties of commutative semigroups of nonexpansive mappings defined
on $\tau $-compact convex subsets of a Banach space $E$ with respect to a
Hausdorff topology $\tau $ on $E.$ For more information on fixed point
theory for nonexpansive mappings, confer e.g., \cite{GoKi} or \cite{GrDu}.

The following fixed point theorem for a commuting family of
nonexpansive mappings was proved in \cite{BoWi}.

\begin{theorem}
\label{firstPub} Suppose that $C$ is a nonempty weak$^{\ast }$-compact
convex subset of a dual Banach space and $\mathcal{S}$ is an arbitrary
family of commuting weak$^{\ast }$-continuous nonexpansive self-mappings on $%
C$. Then $\Fix\mathcal{S}$, the set of the common fixed points of the family
$\mathcal{S}$, is a nonexpansive retract of $C$.
\end{theorem}

Recall that a non-void set $D\subset C$ is a nonexpansive retract of $C$ if
there exists a nonexpansive mapping $R:C\rightarrow D$ such that the
restriction of $R$ to $D$ is the identity. Then $R$ is called a nonexpansive
retraction. The above theorem enabled us to confirm a special case of the
long-standing open problem originally posed by A. T.-M. Lau. (check \cite%
{La2}, \cite{La} or \cite[Question 1]{LaZh} for details).

It is known that the norm is lower semicontinuous in the weak$^{\ast }$
topology. Given a Hausdorff topology $\tau $, the $\tau $-lower
semicontinuity of the norm means that
\begin{equation*}
\norm{\tau\mhyp\lim_{\alpha}x_{\alpha}}\leq \liminf_{\alpha }%
\norm{x_{\alpha}}
\end{equation*}%
for any $\tau $-convergent net $(x_{\alpha })_{\alpha }\subset E.$ The other
known examples of such topologies are weak and strong (norm) topologies. It
turns out that Theorem \ref{firstPub} can be generalized to topologies under
which the norm becomes lower semicontinuous. This paper provides such a
generalization and--in the process--other related ones. But first, we also
show some theorems that follow from Theorem \ref{firstPub}.

\section{Results}

We start with the statement that bears some resemblance to a remarkable
theorem of Bader, Gelander, and Monod \cite{BGM}.

\begin{theorem}
\label{weakTschebyshev}Let $\mathcal{S}$ be a commuting family of weak$%
^{\ast }$-continuous nonexpansive mappings defined on a dual Banach space $E$%
. Let furthermore $\mathcal{S}$ preserve a bounded set $A$ (i.e., $TA=A$ for
all $T\in \mathcal{S}$). Then there is a common fixed point of $\mathcal{S}$
which is located in $\mathcal{C}(A)$, the Tchebyshev center of $A$.
\end{theorem}

Before the proof, let us recall the definition of the Tchebyshev center:

\begin{equation*}
\mathcal{C}(A)=\{c\in E:A\subset \mathbb{B}(c,r(A))\},
\end{equation*}
where $\mathbb{B}$ denotes the closed ball, and $r(A)$ is called the
Tchebyshev radius of the set $A$:
\begin{equation*}
r(A)=\inf\{r\geq0:\exists_{x\in E} A\subset \mathbb{B}(x,r)\}.
\end{equation*}

\begin{proof}
In the dual space $E$, it is known that $\mathcal{C}(A)$ is  nonempty,
bounded, convex and $w^{\ast }$-compact (cf. e.g., \cite{BGM}). So the only
thing to do is to show the $\mathcal{S}$-invariance of it, and use Theorem %
\ref{firstPub}. Take $T\in \mathcal{S}$. If $r$ denotes the Tchebyshev
radius of $A$ and $c\in \mathcal{C}(A)$, then from the fact that $A\subset
\mathbb{B}(c,r)$, we can derive
\begin{equation*}
A=TA\subset T\mathbb{B}(c,r)\subset \mathbb{B}(Tc,r).
\end{equation*}%
The latter inclusion is true because if $z\in T\mathbb{B}(c,r)$, then
obviously there exists $x\in \mathbb{B}(c,r)$ such that $z=Tx$, meaning
\begin{equation*}
\norm{Tc-z}=\norm{Tc-Tx}\leq \norm{c-x}\leq r,
\end{equation*}%
that is, $z\in \mathbb{B}(Tc,r)$. Thus we have $A\subset \mathbb{B}(Tc,r)$,
meaning $Tc\in \mathcal{C}(A)$. From the freedom of choice of $c$ and $T$ we
conclude that $\mathcal{C}(A)$ is $\mathcal{S}$-invariant, which--as stated
before--proves the theorem.
\end{proof}

Using the Kuratowski-Zorn lemma, we can derive the following lemma.

\begin{lemma}
\label{subsurjectivity} For every commuting family $\mathcal{S}$ of
continuous mappings defined on a compact set $C$, there exists a compact
subset of $C$ on which every mapping in $\mathcal{S}$ is surjective.
\end{lemma}

The property from the above lemma was called the subsurjectivity in \cite%
{BoWi2}. Now we can state

\begin{theorem}
\label{smwhrComm} Let $\mathcal{S}$ be a commuting family of weak$^{\ast }$%
-continuous nonexpansive mappings defined on a dual Banach space $E$.
Suppose there exists a $w^{\ast }$-compact, $\mathcal{S}$-invariant set $C\subset E$. Then
the family $\mathcal{S}$ has a common fixed point.
\end{theorem}

\begin{proof}
Use lemma \ref{subsurjectivity} to obtain a set $B\subset C$ with the same
properties as $C$ but on which $\mathcal{S}$ is surjective. Since $B$ is
bounded (because of $w^{\ast }$-compactness), using theorem \ref%
{weakTschebyshev} yields a common fixed point of $\mathcal{S}$ in $E$.
\end{proof}

Note, since the Tchebyshev center property is not monotone ($A\subset B%
\centernot\implies \mathcal{C}(A)\subset \mathcal{C}(B)$), we cannot state
this time that a fixed point of $\mathcal{S}$ is located in $\mathcal{C}(C)$.
Let us show one more example of an argument that utilizes the concept of
subsurjectivity. To this end, we will make use of the following definition.

\begin{definition}
For a family $\mathcal{S}$ of mappings defined on a set $C$, let $\gamma
\mathcal{S}$ denote the maximal subset of $C$, on which all elements from
the semigroup generated by $\mathcal{S}$ commute. If $\gamma \mathcal{S}%
\not=\emptyset $, we will say that $\mathcal{S}$ is somewhere commuting (on $%
C$).
\end{definition}

Let us list basic properties of the $\gamma $ operation.

\begin{lemma}
\label{lemmaInv}The following claims are true:

\begin{enumerate}
\item the set $\gamma \mathcal{S}$ is $\mathcal{S}$-invariant (i.e. $T\in
\mathcal{S}\implies T(\gamma \mathcal{S})\subset \gamma \mathcal{S}$),\label%
{claim1}

\item $Fix\mathcal{S}\subset \gamma \mathcal{S}$,\label{claim2}

\item if the mappings from $\mathcal{S}$ are continuous in some topology,
then $\gamma \mathcal{S}$ is closed in it\label{claim3}.
\end{enumerate}
\end{lemma}

\begin{proof}
Claim (\ref{claim1})\newline
For a fixed $T\in \mathcal{S}$, take $x\in \gamma \mathcal{S}$. Then for any
$T_{1},T_{2}\in \mathcal{S}$,
\begin{equation*}
T_{1}T_{2}Tx=T_{1}(T_{2}T)x=T_{2}(TT_{1}x)=T_{2}T_{1}Tx
\end{equation*}%
and from the maximality of $\gamma \mathcal{S}$, $Tx$ must be also the
element of $\gamma \mathcal{S}$. \newline
Claim (\ref{claim2})\newline
If $T_{1},\,T_{2}\in \mathcal{S}$ and $x\in Fix\mathcal{S}$, then obviously
$T_{1}T_{2}x=x=T_{2}T_{1}x$
and $x\in \gamma \mathcal{S}$, by the maximality argument. \newline
Claim (\ref{claim3})\newline
Take a net $(x_{\alpha })\subset \gamma\mathcal{S}$ with $\lim_{\alpha }x_{\alpha }=x$. Then
the mappings from $\mathcal{S}$ commute over element $x$:
\begin{equation*}
T_{1}T_{2}x=T_{1}T_{2}\lim_{\alpha }x_{\alpha }=\lim_{\alpha
}T_{1}T_{2}x_{\alpha }=\lim_{\alpha }T_{2}T_{1}x_{\alpha
}=T_{2}T_{1}\lim_{\alpha }x_{\alpha }=T_{2}T_{1}x.
\end{equation*}%
And again, we can use the maximality of $\gamma \mathcal{S}$ to infer the
desired result.
\end{proof}

Note that the above lemma can be stated shortly: if $\mathcal{S}$ is a
family of continuous mappings commuting at least at one point, then there is
a closed $\mathcal{S}$-invariant superset of $Fix\mathcal{S}$, on which $%
\mathcal{S}$ commutes.

\begin{lemma}
If $\mathcal{S}$ is a family of continuous mappings on a compact set $C$
which commute at least over one point, then there exists a compact subset of
$C$ on which $\mathcal{S}$ consists of commuting surjections.
\end{lemma}

\begin{proof}
Let
$
\mathcal{S}|_{C}=\{T|_{C}:T\in \mathcal{S}\}.
$
With the aid of Lemma \ref{lemmaInv}, we see that $A=\gamma (\mathcal{S}%
|_{C})$ is a nonempty, $\mathcal{S}$-invariant, compact set on which $%
\mathcal{S}$ commutes. Then use Lemma \ref{subsurjectivity} to obtain
another set $B\subset A$ with the same properties, but on which $\mathcal{S}$
is surjective.
\end{proof}

As it is known (cf. \cite{Kirk} and \cite[p. 61]{BaRo} for the proof):

\begin{theorem}[Freudenthal-Hurewicz]
A nonexpansive surjection on a compact metric space is an isometry.
\end{theorem}

We obtain instantly

\begin{theorem}
If $\mathcal{S}$ is a family of nonexpansive mappings, commuting somewhere
(and defined) on the compact metric space $C$, then there exists a compact
set $A\subset C$ on which $\mathcal{S}$ are surjective commuting isometries.
\end{theorem}

Thus $\mathcal{S}|_{A}$ generates the abelian group of isometries.
Now, let us deal with the generalizations related to Theorem \ref{firstPub}.

\begin{theorem}
\label{infThm} Let $C$ be a $\tau $-compact subset of a normed linear space $%
E$, where the norm is $\tau $-lower semicontinuous ($\tau $-LS) and let $%
\mathcal{S}$ be a family of $\tau $-continuous mappings on $C$ with the
following property: for every finite subfamily $\mathcal{A}\subset \mathcal{S%
}$, $Fix\mathcal{A}$ is a nonexpansive retract of $C$. Then $Fix\mathcal{S}$
is also a nonexpansive retract of $C$.
\end{theorem}

\begin{proof}
We may assume that $\mathcal{S}$ is infinite. For later use, notice that for
$T\in \mathcal{S}$ the set $FixT$ is $\tau $-closed. Consider
\begin{equation*}
\Lambda =\{\alpha \subset \mathcal{S}:\#\alpha <\infty \}
\end{equation*}%
as the directed set with the order relation $\leq $ as inclusion. Denote by $%
R_{\alpha }$ a nonexpansive retraction from $C$ to $Fix_{\alpha
}=\bigcap_{T\in \alpha }FixT$, whose existence follows from the assumptions.
This construction gives us a net $(R_{\alpha })_{\alpha \in \Lambda }$ in
the $\tau $-compact space $C^{C}$ (Tychonoff's theorem), and we can select a
$\tau $-convergent subnet $(R_{\alpha _{\gamma }})_{\gamma \in \Gamma }$.
Using the definition of the compact topology, it is easy to deduce that we
can define a mapping $R$ such that for every $x\in C$,
\begin{equation*}
Rx=\tau \mhyp\lim_{\gamma }R_{\alpha _{\gamma }}x.
\end{equation*}%
For $T\in \mathcal{S}$, take $\gamma _{0}$ such that for every $\gamma \geq
\gamma _{0},$ we have $\alpha _{\gamma }\geq \{T\}$. It follows
straightforward from subnet's definition. Then for any $x\in C,$
\begin{equation*}
\forall _{\gamma \geq \gamma _{0}}R_{\alpha _{\gamma }}x\in Fix_{\alpha
_{\gamma }}\subset Fix_{\alpha _{\gamma _{0}}}\subset FixT
\end{equation*}%
and hence, $R_{\alpha _{\gamma }}x$ lies eventually in the $\tau $-closed
set $FixT$. That is, $Rx\in FixT$, and from the freedom of choice of $T$ and
$x$, we have $RC\subset Fix\mathcal{S}$. But also
\begin{equation*}
x\in Fix\mathcal{S}\implies x\in Fix_{\alpha }\implies R_{\alpha
}x=x\implies Rx=x\implies x\in FixR.
\end{equation*}%
Hence
$
RC\subset Fix\mathcal{S}\subset FixR\subset RC
$,
which shows that $R$ is a retraction from $C$ to $Fix\mathcal{S}$. It
remains to use the $\tau $-lower semicontinuity to prove that $R$ is
nonexpansive:
\begin{equation}
\norm{Rx-Ry}=\norm{\tau\mhyp\lim_{\gamma}R_{\alpha_{\gamma}}x-\tau\mhyp%
\lim_{\gamma}R_{\alpha_{\gamma}}y}\leq \liminf_{\gamma }\norm{R_{\alpha_{%
\gamma}}x-R_{\alpha_{\gamma}}y}\leq \norm{x-y}.  \label{nonExpLim}
\end{equation}
\end{proof}

\begin{remark}
\label{remarkModi} The above theorem can be easily modified at least in a
few ways:

\begin{enumerate}
\item we can replace the $\tau $-continuity of mappings from $\mathcal{S}
$ with an assumption that their fixed point sets are $\tau $-closed,

\item the proof can be stated in metric spaces, assuming the following
definition of $\tau $-LS
\begin{equation*}
d(\tau \mhyp\lim_{\alpha }x_{\alpha },0)\leq \liminf_{\alpha }d(x_{\alpha
},0),
\end{equation*}

\item let $\mathcal{S}$ be a family of continuous mappings defined on a
compact set $C$. If for every finite subfamily $\mathcal{A}\subset \mathcal{S%
}$ there exists a retraction $C\rightarrow Fix\mathcal{A}$, then the same
holds for the whole family $\mathcal{S}$,

\item \label{remarkFirlmly}Recall the firmly nonexpansive mapping $T:C\to C$ is defined by the following equation $$\norm{Tx-Ty}\leq\norm{a(x-y)+(1-a)(Tx-Ty)}$$, which needs to hold for every $a\in (0,1)$. Then notice that if $\tau $ denotes the strong topology and every
$Fix\mathcal{A}$ is a firmly nonexpansive retract, then $Fix\mathcal{S}$ is
also a firmly nonexpansive retract- rewrite equation (\ref{nonExpLim}%
) from the end of the proof (with $a\in (0,1)$):

\begin{equation*}
\norm{Rx-Ry}=\norm{\lim_{\gamma}R_{\alpha_{\gamma}}x-\lim_{\gamma}R_{%
\alpha_{\gamma}}y}=
\end{equation*}%
\begin{equation*}
\lim_{\gamma }\norm{R_{\alpha_{\gamma}}x-R_{\alpha_{\gamma}}y}\leq
\lim_{\gamma }\norm{a(x-y)+(1-a)(R_{\alpha_{\gamma}}x-R_{\alpha_{\gamma}}y)}=
\end{equation*}%
\begin{equation*}
\norm{a(x-y)+(1-a)(\lim_{\gamma}R_{\alpha_{\gamma}}x-\lim_{\gamma}R_{%
\alpha_{\gamma}}y)}=\norm{a(x-y)+(1-a)(Rx-Ry)},
\end{equation*}

\item we can change a 'nonexpansive retract' to a '$\tau $-continuous affine
retract' (i.e. retract, for which there exists a $\tau $-continuous affine
retraction) in the assumptions. As a result, we get an affine retract this
time. The $\tau $-LS assumption is superfluous. Again, we only need to
modify the end of the original proof.
\end{enumerate}
\end{remark}

As we can see, Theorem (\ref{infThm}) may have many variants but in order to
fuel them, the finite case theorem is always needed. Otherwise, they may be
vacuously true. One source of such supporting theorems is the method used by
Bruck \cite{Br2}, and recently by Saedi \cite{SaIs}. To extract it in a
general way, for the readability, let us provide the following
definition.

\begin{definition}
For a given family $\mathcal{S}$ of mappings defined on a set $C$, let $%
\mathcal{R}(\mathcal{S})$ denote the set of all retractions from $C$ to $Fix%
\mathcal{S}$. For the singleton we use the shorthand $\mathcal{R}(\{T\})=%
\mathcal{R}(T)$.
\end{definition}

Since we will use it later, we give without proof the following simple lemma
(cf. \cite{BoWi}).

\begin{lemma}
\label{commute_retract} Let $\mathcal{S}$ be a family of mappings and
suppose there exists a retraction $R$ onto $Fix\mathcal{S}$. If $T$ commutes
with every member of the family $\mathcal{S}$, then
$
FixT\cap Fix\mathcal{S}=FixTR.
$
\end{lemma}

\begin{theorem}
\label{bruck} Let a semigroup $\mathcal{A}\subset C^{C}$ has the following
property:
\begin{equation}
\forall _{T\in \mathcal{A}}\mathcal{R}(T)\cap \mathcal{A}\not=\emptyset .
\label{bruck_ass}
\end{equation}%
Then for every finite commuting family $\mathcal{S}\subset
\mathcal{A}$ we have also
\begin{equation*}
\mathcal{R}(\mathcal{S})\cap \mathcal{A}\not=\emptyset .
\end{equation*}
\end{theorem}

\begin{proof}
If family $\mathcal{S}$ is a singleton, then theorem is true directly from the assumptions.
Let the theorem be true for the families with $n$ elements.
Consider $\mathcal{S}_{n+1}=\{T_{1},\ldots,T_{n+1}\}$ and it's subfamily $\mathcal{S}_{n}=\{T_{1},\ldots,T_{n}\}$.
Then from the induction hypothesis exists retraction $R_{n}\in \mathcal{R}(%
\mathcal{S}_{n})\cap \mathcal{A}$. Since $T_{n+1}R_{n}\in \mathcal{A}$, there is also in $%
\mathcal{A}$ a retraction $R_{n+1}$ from $C$ to $Fix(T_{n+1}R_{n})$. Lemma %
\ref{commute_retract} gives $Fix(T_{n+1}R_{n})=Fix\mathcal{S}_{n+1}$, so $%
R_{n+1}$ belongs also to $\mathcal{R}(\mathcal{S}_{n+1})$. 
This proves the theorem via mathematical induction.
\end{proof}

The idea is that $\mathcal{A}$ represents a well behaved property of
mappings, e.g. nonexpansivity, isometricity, being affine, and so on.
Nevertheless, the presented reasoning would fail if the set of interest from
Theorem \ref{bruck}--the family $\mathcal{A}$--would not posses the
semigroup structure or, more specifically, we could not know whether $%
T_{n+1}R_{n}\in \mathcal{A}$. In those cases we must look for the alternate
ways of generating the appropriate retraction $R_{n+1}$. One of such
arguments will be presented below.

\begin{lemma}
\label{approx_fix} Let $\mathcal{S}$ be a commuting family of nonexpansive
mappings and suppose there exists a nonexpansive retraction $R$ onto $Fix%
\mathcal{S}$. Suppose that a nonexpansive $T$ commutes with members of $%
\mathcal{S}$. Then every approximate fixed point sequence $(x_{n})$ of $TR$
is also the approximate fixed point sequence of both the family $\mathcal{S}$
and the mapping $T$.
\end{lemma}

\begin{proof}
Note the following fact: if $(x_{n})$ is an approximate fixed
point sequence of the mapping $Q$, and for a nonexpansive $P$ we have $%
PQx_{n}=Qx_{n}$, then $(x_{n})$ is also an approximate fixed point sequence
of $P$. Indeed
\begin{equation*}
\norm{Px_{n}-x_{n}}\leq \norm{Px_{n}-Qx_{n}}+\norm{Qx_{n}-x_{n}}=
\end{equation*}%
\begin{equation*}
\norm{Px_{n}-PQx_{n}}+\norm{Qx_{n}-x_{n}}\leq 2\norm{Qx_{n}-x_{n}}%
\rightarrow 0.
\end{equation*}%
Now, notice that for any $S\in \mathcal{S}$ from the lemma we have
\begin{equation}
STRx_{n}=TSRx_{n}=TRx_{n}  \label{long_eq}
\end{equation}%
so if we denote $Q=TR$ and $P=S$, we obtain
$
\norm{Sx_{n}-x_{n}}\rightarrow 0.
$
From the equation (\ref{long_eq}) we also conclude that $TRx_{n}$ is a
common fixed point of the family $\mathcal{S}$, which gives
$
RTRx_{n}=TRx_{n}.
$
Now, letting $Q=TR$ and $P=R,$ yields
$
\norm{Rx_{n}-x_{n}}\rightarrow 0.
$
It follows that
\begin{equation*}
\norm{Tx_{n}-x_{n}}\leq \norm{Tx_{n}-TRx_{n}}+\norm{TRx_{n}-x_{n}}%
\rightarrow 0.
\end{equation*}
\end{proof}

\begin{theorem}
\label{proofFiniteRetract}Let $C$ be a nonempty $\tau $-compact convex and
bounded subset of a Banach space such that its norm is $\tau $-LS. Let $%
\mathcal{S}$ be a finite commuting family of nonexpansive $\tau $-continuous
self mappings on $C$. Then $Fix\mathcal{S}$ is a nonempty nonexpansive
retract of $C$.
\end{theorem}

\begin{proof}
Let us impose one more restriction on $\mathcal{S}$: it has to possess identity as an element.
It is easy to see, that if we would have proof for such families, then the theorem would also follow for the unrestricted ones.
So let's continue with the restricted families.
Then if $\mathcal{S}$ is a singleton, the theorem is trivially true: the
retraction we are looking for is just the identity on $C$. 

In the spirit of the mathematical induction, let us now assume that there exists a
nonexpansive retraction for the given $\mathcal{S}_{n}=\{T_1,\ldots,T_n\}$:
$$
R_{n}:C\rightarrow Fix\mathcal{S}_{n}.
$$
From that we are going to construct another nonexpansive retraction (for the family $\mathcal{S}_{n+1}=\mathcal{S}_{n}\cup\{T_{n+1}\}$):
$$
R_{n+1}:C\rightarrow Fix\mathcal{S}_{n+1}
$$
which proves our point. 

Note that $C$ is strongly closed: assume $x_{\alpha }\rightarrow x$
strongly. From $\tau $-compactness of $C$, take a subnet $(x_{\alpha
_{\gamma }})_{\gamma }$ that has a $\tau $-limit in $C$. Then
\begin{equation*}
\norm{\tau\mhyp\lim_{\gamma}x_{\alpha_{\gamma}}-x}\leq \liminf_{\gamma }%
\norm{x_{\alpha_{\gamma}}-x}=\lim_{\gamma }\norm{x_{\alpha_{\gamma}}-x}=0.
\end{equation*}%
That is, $x=\tau \mhyp\lim_{\gamma }x_{\alpha _{\gamma }}\in C$. So $C$
is strongly closed, thus complete. Take
\begin{equation*}
T_{x,s}z=\frac{1}{s}x+\left( 1-\frac{1}{s}\right) T_{n+1}R_{n}z.
\end{equation*}%
Since $T_{x,s}$ is a contraction defined on the complete space, there exists
exactly one point $F_{s}x\in C$ such that
$
T_{x,s}F_{s}x=F_{s}x.
$
This defines the mapping $F_{s}:C\rightarrow C$. Note that
\begin{equation*}
\norm{F_{s}x-F_{s}y}=\norm{\frac{1}{s}(x-y)+(1-%
\frac{1}{s})(T_{n+1}R_{n}F_{s}x-T_{n+1}R_{n}F_{s}x)}\leq
\end{equation*}%
\begin{equation*}
\frac{1}{s}\norm{x-y}+(1-\frac{1}{s})\norm{F_{s}x-F_{s}y}
\end{equation*}%
from which we deduce the nonexpansivity of $F_{s}$.
Notice further that the fact
\begin{equation*}
F_{s}x=T_{x,s}F_{s}x=\frac{1}{s}x+(1-\frac{1}{s})T_{n+1}R_{n}F_{s}x
\end{equation*}%
and boundedness of $C$ implies
\begin{equation}
\norm{T_{n+1}R_{n}F_{s}x-F_{s}x}=\frac{1}{s}\norm{T_{n+1}R_{n}F_{s}x-x}\leq
\frac{diamC}{s}\rightarrow 0,  \label{approx_diam}
\end{equation}%
that is, $(F_{s})_{s\in \mathbb{N}}\subset C^{C}$ forms with every $x\in C$
an approximate fixed point sequence of the mapping $T_{n+1}R_{n}$. From
Lemma \ref{approx_fix} we conclude that the same is true for the family $%
\mathcal{S}_{n+1}$. Since $C$ is $\tau $-compact, 
then--from Tychonoff's theorem--$(F_{s})_{s\in \mathbb{N}}$ 
has a convergent subnet (in a pointwise convergence with respect to $\tau $-topology):
\begin{equation}
R_{n+1}=\tau \mhyp\lim_{\alpha }F_{s_{\alpha }}.  \label{retr_limit}
\end{equation}%
Then, for every $T\in \mathcal{S}_{n+1}$,
\begin{equation*}
\norm{TR_{n+1}x-R_{n+1}x}=\norm{T(\tau\mhyp\lim_{\alpha}F_{s_{\alpha}}x)-\tau%
\mhyp\lim_{\alpha}F_{s_{\alpha}}}x\leq \liminf_{\alpha }\norm{TF_{s_{%
\alpha}}x-F_{s_{\alpha}}x}\rightarrow 0,
\end{equation*}%
where convergence is the mentioned consequence of Lemma \ref{approx_fix}. This
gives $TR_{n+1}x=R_{n+1}x$, and from the freedom of choice of $x$ and $T$ we have
$
R_{n+1}C\subset FixS_{n+1}.
$
Observe that%
\begin{equation*}
x\in FixT_{n+1}R_{n}\implies T_{x,s}x=x\implies F_{s}x=x\implies R_{n+1}x=x
\end{equation*}%
which means
$
FixT_{n+1}R_{n}\subset FixR_{n+1}.
$
From Lemma \ref{commute_retract}, follows
\begin{equation*}
FixR_{n+1}\subset R_{n+1}C\subset Fix\mathcal{S}_{n+1}\subset
FixT_{n+1}R_{n}\subset FixR_{n+1},
\end{equation*}%
from which we see
$
R_{n+1}C=Fix\mathcal{S}_{n+1}=FixR_{n+1}.
$
So, $R_{n+1}$ is a retraction onto $Fix\mathcal{S}_{n+1}$. Its
nonexpansivness follows from the fact that $F_{s}$ is nonexpansive for every
$s,$ and from $\tau $-LS property (check the eq. (\ref%
{nonExpLim})).
\end{proof}

Now, having in mind Theorem \ref{infThm}, we have obviously the following
result.

\begin{theorem}
Let $C$ be a $\tau $-compact, convex and bounded subset of a Banach space
with $\tau $-LS norm. Then the set of common fixed points of a commuting
family of $\tau $-continuous nonexpansive mappings on $C$ is a nonempty
nonexpansive retract of $C$.
\end{theorem}

Notice that we can infer also (compare \cite[Theorem 3]{Bruck2})

\begin{theorem}
Let $C$ be a nonempty compact convex subset of a Banach space. Then the
fixed point set of a commuting family of nonexpansive mappings on $C$ is a
nonempty firmly nonexpansive retract of $C$.
\end{theorem}

\begin{proof}
It is known that the mappings $F_{s}$ defined in the proof of Theorem \ref%
{proofFiniteRetract} are firmly nonexpansive(cf. \cite[Theorem 11.3]{GoKi}). It means that also their
strong limit would be firmly nonexpansive. In effect, the `strong
topology-version' of Theorem \ref{proofFiniteRetract} gives us firmly
nonexpansive retracts. Then it suffices to use Remark \ref{remarkModi},
(\ref{remarkFirlmly}) to finish the proof.
\end{proof}

\end{document}